\documentclass[11pt]{article}      
\usepackage{amsmath,amssymb,amscd,amsthm, graphicx, verbatim, setspace} 
\usepackage{authblk}

\theoremstyle{plain}
\newtheorem{thm}{Theorem}[section]
\newtheorem{lem}[thm]{Lemma}
\newtheorem{cor}[thm]{Corollary}

\newcommand{\edim}{\operatorname{edim}}

\title{Metric dimension and pattern avoidance in graphs}
\date{}
\author{Jesse Geneson\\
\small\tt geneson@gmail.com
}
\begin{document}
\maketitle
\begin{abstract}
In this paper, we prove a number of results about pattern avoidance in graphs with bounded metric dimension or edge metric dimension. We show that the maximum possible number of edges in a graph of diameter $D$ and edge metric dimension $k$ is at most $(\lfloor \frac{2D}{3}\rfloor +1)^{k}+k \sum_{i = 1}^{\lceil \frac{D}{3}\rceil } (2i)^{k-1}$, sharpening the bound of $\binom{k}{2}+k D^{k-1}+D^{k}$ from Zubrilina (2018). We also show that the maximum value of $n$ for which some graph of metric dimension $\leq k$ contains the complete graph $K_{n}$ as a subgraph is $n = 2^{k}$. We prove that the maximum value of $n$ for which some graph of metric dimension $\leq k$ contains the complete bipartite graph $K_{n,n}$ as a subgraph is $2^{\Theta(k)}$. Furthermore, we show that the maximum value of $n$ for which some graph of edge metric dimension $\leq k$ contains $K_{1,n}$ as a subgraph is $n = 2^{k}$. We also show that the maximum value of $n$ for which some graph of metric dimension $\leq k$ contains $K_{1,n}$ as a subgraph is $3^{k}-O(k)$. 

In addition, we prove that the $d$-dimensional grids $\prod_{i = 1}^{d} P_{r_{i}}$ have edge metric dimension at most $d$. This generalizes two results of Kelenc et al. (2016), that non-path grids have edge metric dimension $2$ and that $d$-dimensional hypercubes have edge metric dimension at most $d$. We also provide a characterization of $n$-vertex graphs with edge metric dimension $n-2$, answering a question of Zubrilina. As a result of this characterization, we prove that any connected $n$-vertex graph $G$ such that $\edim(G) = n-2$ has diameter at most $5$. More generally, we prove that any connected $n$-vertex graph with edge metric dimension $n-k$ has diameter at most $3k-1$.
\end{abstract}

\section{Introduction}
Metric dimension is a graph parameter with many applications including robot navigation, chemistry, pattern recognition, image processing, and locating intruders in networks \cite{kh, ch, pr, li, h}. In particular for the application to robot navigation \cite{kh}, the robot is assumed to be moving from vertex to vertex in a graph, where some of the vertices are distinguished as landmarks. The robot is able to determine its distance to each landmark, and it uses those distances to determine its location in the graph. The goal is to use as few landmarks as possible, and the number of landmarks used is the metric dimension.

Recently, a parameter closely related to metric dimension was defined in \cite{kel}. Imagine that the robot now moves from edge to edge instead of vertex to vertex. The robot can still determine its distance from the landmarks, where the distance from an edge $\left\{u, v\right\}$ to a landmark is the minimum of the distance from $u$ or $v$ to the landmark. Again the goal is to use as few landmarks as possible, and that number is the edge metric dimension. 

A number of results relating pattern avoidance and metric dimension have been proved in the last few decades, but perhaps the most famous one is the result of Khuller et al. \cite{kh} that no graph of metric dimension $2$ contains $K_5$ as a subgraph (and more generally, that graphs of metric dimension $k$ cannot have $K_{2^k+1}$ as a subgraph). More generally, Sudhakara and Kumar proved an upper bound of $(D+1)^2$ on the number of vertices in subgraphs of diameter $D$ in graphs of metric dimension $2$ \cite{waset}.

Besides pattern avoidance, researchers have also investigated the maximum size of graphs with a given diameter and metric dimension. In particular, Hernando et al. \cite{exmd} proved that the maximum possible number of vertices in a graph of diameter $D$ and metric dimension $k$ is at most $(\lfloor \frac{2D}{3}\rfloor +1)^{k}+k \sum_{i = 1}^{\lceil \frac{D}{3}\rceil } (2i-1)^{k-1}$. Kelenc et al. \cite{kel} proved that graphs with edge metric dimension $k$ and diameter $D$ have at most $(D+1)^k$ edges. Later, Zubrilina sharpened the bound for edge metric dimension to $\binom{k}{2}+k D^{k-1}+D^k$ \cite{zu}. 

Kelenc et al. \cite{kel} also bounded the edge metric dimension of several classes of graphs, including $2$-dimensional grid graphs and $d$-dimensional hypercube graphs. They asked for a characterization of all graphs with edge metric dimension $n-1$, which Zubrilina found \cite{zu}. Using the characterization, Zubrilina proved that such graphs have diameter at most $2$. Furthermore, Zubrilina asked for a characterization of all graphs of edge metric dimension $n-2$.

In this paper, we prove several new bounds related to pattern avoidance, metric dimension, and diameter. We also answer Zubrilina's question and show that $n$-vertex graphs with edge metric dimension $n-2$ have diameter at most $5$. Before stating more details about our results, we first discuss some terminology.

\subsection{Terminology}
All graphs in this paper will be simple and connected. Let $d(u, v)$ denote the distance between $u$ and $v$ for any vertices $u, v \in V(G)$. The \emph{distance vector} $d_{v, S}$ of a vertex $v$ with respect to a subset of vertices $S \subset V(G)$ is the vector with $|S|$ coordinates which has a single coordinate with value $d(x, v)$ for each vertex $x \in S$. A set of vertices $S \subset V(G)$ is a \emph{resolving set} for the vertices of $G$ if no two vertices in $V(G)$ have the same distance vector with respect to $S$. The \emph{metric dimension} $\dim(G)$ is the minimum size of a resolving set for $G$, and a resolving set $S$ for the vertices of $G$ is called a \emph{metric basis} if $S$ contains exactly $\dim(G)$ elements.

For any vertex $v \in V(G)$ and edge $e \in E(G)$ such that $e = \left\{x, y \right\}$, let $d(e, v) = \min(d(x, v), d(y, v))$. The distance vector $d_{e, S}$ of an edge $e$ with respect to a subset of vertices $S \subset V(G)$ is the vector with $|S|$ coordinates which has a single coordinate with value $d(x, e)$ for each vertex $x \in S$. A set of vertices $S \subset V(G)$ is a resolving set for the edges of $G$ if no two edges in $E(G)$ have the same distance vector with respect to $S$. The edge metric dimension $\edim(G)$ is the minimum size of a resolving set for the edges of $G$, and a resolving set $S$ for the edges of $G$ is called a metric basis if $S$ contains exactly $\edim(G)$ elements.

The chromatic number of a graph $G$ is the minimum number of colors required to label the vertices of $G$ so that no edge has two vertices of the same color. The degeneracy of a graph $G$ is the minimum $r$ such that every subgraph of $G$ has a vertex of degree at most $r$. We will call $x$ a non-mutual neighbor of vertices $u, v$ if $x$ is adjacent to $u$ or $v$, but not both.

\subsection{New results}

In Section \ref{btdd}, we prove that the maximum possible number of edges in a graph of diameter $D$ and edge metric dimension $k$ is at most $(\lfloor \frac{2D}{3}\rfloor +1)^{k}+k \sum_{i = 1}^{\lceil \frac{D}{3}\rceil } (2i)^{k-1}$. This sharpens the bound of $\binom{k}{2}+k D^{k-1}+D^{k}$ from Zubrilina \cite{zu}, which improved on the bound of Kelenc et al. \cite{kel}. 

In Section \ref{sfpc}, we derive several results about specific forbidden subgraphs in graphs of metric dimension or edge metric dimension $\leq k$. We prove that the maximum value of $n$ for which some graph of metric dimension $\leq k$ contains $K_{n}$ as a subgraph is $n = 2^{k}$. We also prove that the maximum value of $n$ for which some graph of metric dimension $\leq k$ contains $K_{n,n}$ as a subgraph is $2^{\Theta(k)}$, and that the maximum value of $n$ for which some graph of edge metric dimension $\leq k$ contains $K_{n,n}$ as a subgraph is $2^{\Theta(k)}$. 

In the same section, we show that the maximum value of $n$ for which some graph of edge metric dimension $\leq k$ contains $K_{1,n}$ as a subgraph is $n = 2^{k}$. We also show that the maximum value of $n$ for which some graph of metric dimension $\leq k$ contains $K_{1,n}$ as a subgraph is between $3^{k}-k-1$ and $3^k-1$. Using these bounds, we prove that the maximum possible chromatic number and degeneracy of a graph of metric dimension $\leq k$ are $2^{\Theta(k)}$.

In Section \ref{mr}, we prove that the $d$-dimensional grids $\prod_{i = 1}^{d} P_{r_{i}}$ have edge metric dimension at most $d$. This bound generalizes the results of Kelenc et al. (2016) that non-path grids have edge metric dimension $2$ and that $d$-dimensional hypercubes have edge metric dimension at most $d$. In addition, we provide a characterization of $n$-vertex graphs with edge metric dimension $n-2$. This answers a question of Zubrilina, and as a result of this characterization, we prove that any connected $n$-vertex graph $G$ such that $\edim(G) = n-2$ has diameter at most $5$. In general, we prove that any connected $n$-vertex graph with edge metric dimension $n-k$ has diameter at most $3k-1$.

\section{Bounds in terms of diameter and dimension}\label{btdd}

The next proof is very similar to the proof in \cite{exmd} that the maximum possible order of a graph of diameter $D$ and metric dimension $k$ is equal to $(\lfloor \frac{2D}{3}\rfloor +1)^{k}+k \sum_{i = 1}^{\lceil \frac{D}{3}\rceil } (2i-1)^{k-1}$. 

\begin{thm}
Let $G$ be a graph of diameter $D$ and edge metric dimension $k$. Then $|E(G)| \leq (\lfloor \frac{2D}{3}\rfloor +1)^{k}+k \sum_{i = 1}^{\lceil \frac{D}{3}\rceil } (2i)^{k-1}$.
\end{thm}

\begin{proof}
Let $S$ be a metric basis for the edges of $G$ and let $c \in [0, D]$ be an integer that will be chosen at the end. For each $v \in S$ and $i \in [0, c]$, define $N_{i}(v) = \left\{e \in E(G) : d(e, v) = i \right\}$. First note that $|d(e, u)-d(f, u)| \leq 2i+1$ for any two edges $e, f \in N_{i}(v)$ and any vertex $u \in S$. Thus for any edge $e \in N_{i}(v)$, we have that $d(e, t)$ has at most $2i+2$ possible values for each $t \in S$ such that $t \neq v$. Thus $|N_{i}(v)| \leq (2i+2)^{k-1}$. 

Consider $e \in E(G)$ such that $e \not \in N_{i}(v)$ for all $i \in [0, c]$ and $v \in S$, i.e., $d(e,v) \geq c+1$ for all $v \in S$. Since $d_{e, S}$ only has entries between $c+1$ and $D$ inclusive, there are at most $(D-c)^k$ such edges. Since every edge is either in $N_{i}(v)$ for some $i \in [0, c]$ and $v \in S$, or at least distance $c+1$ from every vertex $v \in S$, we have

\[ |E(G)| \leq (D-c)^{k}+k \sum_{i = 0}^{c} (2i+2)^{k-1}. \]

Setting $c = \lceil \frac{D}{3} \rceil -1$ gives the upper bound.
\end{proof}

We use both of the theorems below for the pattern avoidance results in the next section. The first result generalizes the proof for the $k = 2$ case from \cite{waset}. 

\begin{thm}\label{fordiam}
Let $G$ be a graph with $\dim(G) = k$, and let $H$ be a subgraph of $G$ with diameter $D$. Then $|V(H)| \leq (D+1)^{k}$.
\end{thm}

\begin{proof}
Let $S$ be a metric basis for $V(G)$. For each vertex $v$ in $H$, each coordinate of $d_{v, S}$ has at most $D+1$ possible values, since $H$ has diameter $D$. Thus there are at most $(D+1)^{k}$ vertices in $H$. 
\end{proof}

\begin{thm}\label{fordiame}
Let $G$ be a graph with $\edim(G) = k$, and let $H$ be a subgraph of $G$ with diameter $D$. Then $|E(H)| \leq (D+1)^k$.
\end{thm}

\begin{proof}
Let $S$ be a metric basis for $E(G)$. For each edge $e$ in $H$, each coordinate of $d_{e, S}$ has at most $D+1$ possible values, since $H$ has diameter $D$. Thus there are at most $(D+1)^{k}$ edges in $H$. 
\end{proof}

\section{Specific forbidden patterns, chromatic number, and degeneracy}\label{sfpc}

The upper bound in the next theorem is well known \cite{kh}. We prove a matching lower bound to show that it is sharp. The graph in the lower bound construction also appears in \cite{zu}.

\begin{thm}\label{mdcomp}
The maximum possible value of $n$ for which some graph of metric dimension $\leq k$ contains $K_{n}$ as a subgraph is $n = 2^{k}$.
\end{thm}

\begin{proof}
The upper bound was proved in \cite{kh}, and it also follows from Theorem \ref{fordiam}.

For the lower bound, define $G$ to be the graph obtained from $K_{2^{k}}$ by adding $k$ vertices with edges defined as follows: For each vertex $v$ in the copy of $K_{2^{k}}$, label $v$ with a binary string of length $k$. For each of the new vertices $u_{1}, \dots, u_{k}$, add an edge from $u_{i}$ to $v$ if the $i^{th}$ digit of $v$ is $0$.

First note that the metric dimension of $G$ is at least $k$ since $G$ contains a complete subgraph with $2^{k}$ vertices. If we set $S = \left\{u_{1}, \dots, u_{k} \right\}$, then $d(v, u_i)$ is one more than the $i^{th}$ digit of $v$ for each $1 \leq i \leq k$ and all $v$ in the copy of $K_{2^{k}}$. Moreover each $u_{i}$ is the only vertex $v$ in $G$ with $d(v, u_i) = 0$, so $S$ is a metric basis for $G$.
\end{proof}

Both of the next two results give bounds on the number of edges in graphs of metric dimension or edge metric dimension $k$, as well as corollaries about the chromatic number and degeneracy. In $K_{1, n}$ for $n > 1$, we refer to the vertex of degree $n$ as the center vertex.

\begin{thm}\label{emdstar}
The maximum possible value of $n$ for which some graph of edge metric dimension $\leq k$ contains $K_{1,n}$ as a subgraph is $n = 2^{k}$.
\end{thm}

\begin{proof}
The upper bound was proved in \cite{kel}, but we include the argument for completeness. Let $H$ be a subgraph isomorphic to $K_{1,n}$ of a graph of edge metric dimension $k$. For all edges $e \in H$, each coordinate of $d_{e, S}$ has at most $2$ possible values. Thus $H$ has at most $2^k$ edges. 

For the lower bound, define $G$ to be the graph obtained from $K_{1, 2^{k}}$ with center vertex $c$ by adding $k$ vertices with edges defined as follows: For each non-center vertex $v$ in the copy of $K_{1,2^{k}}$, label $v$ with a binary string of length $k$. For each of the new vertices $u_{1}, \dots, u_{k}$, add an edge from $u_{i}$ to $v$ if the $i^{th}$ digit of $v$ is $0$.

First note that the edge metric dimension of $G$ is at least $k$ since $G$ contains a star with $2^{k}$ non-center vertices. If we set $S = \left\{u_{1}, \dots, u_{k} \right\}$, then $d(\left\{v, c \right\}, u_i)$ is one more than the $i^{th}$ digit of $v$ for each $1 \leq i \leq k$ and all non-center vertices $v$ in the copy of $K_{1, 2^{k}}$. 

All other edges $e \in E(G)$ are adjacent to some $u_{i}$ and thus have $d(e, u_i) = 0$. For each $t \neq i$, $d(e, u_t)$ is one more than the $t^{th}$ digit of $v$, where $v$ is the vertex such that $e = \left\{u_{i},v\right\}$. Thus $S$ is a metric basis for the edges of $G$.
\end{proof}

The next result shows that the bound in Theorem \ref{fordiam} is not far from the exact value for $D = 2$.

\begin{thm}\label{mdstar}
The maximum possible value of $n$ for which some graph of metric dimension $\leq k$ contains $K_{1,n}$ as a subgraph is between $3^{k}-k-1$ and $3^k-1$.
\end{thm}

\begin{proof}
The upper bound follows from Theorem \ref{fordiam}, since $K_{1, n}$ has diameter $2$. 

For the lower bound, define $G$ to be the graph obtained from $K_{1, 3^{k}}$ with center vertex $c$ by adding $2k$ vertices $r_{1}, \dots, r_{k}, s_{1}, \dots, s_{k}$ with edges defined as follows: For each non-center vertex $v$ in the copy of $K_{1,3^{k}}$, label $v$ with a base $3$ string of length $k$. Add an edge from $s_{i}$ to $v$ if the $i^{th}$ digit of $v$ is $0$. Also add an edge from $r_{i}$ to $s_{i}$, and add an edge from $r_{i}$ to $v$ if the $i^{th}$ digit of $v$ is $1$. 

If we set $S = \left\{s_{1}, \dots, s_{k} \right\}$, then $d(v, s_i)$ is one more than the $i^{th}$ digit of $v$ for each $1 \leq i \leq k$ and all non-center vertices $v$ in the copy of $K_{1, 3^{k}}$. Moreover for each $i$, $s_{i}$ is the only vertex $v$ such that $d(v,s_i) = 0$.

Each $r_{i}$ has $d_{r_{i}, S}$ distinct from $d_{r_{j}, S}$ for all $j \neq i$, but it is possible that $d_{r_{i}, S} = d_{v, S}$ for some non-center vertex $v$ in the copy of $K_{1, 3^{k}}$, or that $d_{c, S} = d_{v, S}$ where $c$ is the center vertex in the copy of $K_{1, 3^{k}}$. Note that every coordinate of $d_{c, S}$ is $2$, so $d_{c, S}$ is distinct from $d_{r_{i}, S}$ for each $i$.

Let $G'$ be the graph obtained from $G$ by deleting any non-center vertex $v$ in the copy of $K_{1, 3^{k}}$ such that $d_{c, S} = d_{v, S}$ or $d_{r_{i}, S} = d_{v, S}$ for some $i$. There are at most $k+1$ such vertices $v$, so $G'$ contains a star with at least $3^{k}-k-1$ non-center vertices. By definition, $S$ is a resolving set for $G'$.
\end{proof}

Next we bound the size of complete bipartite subgraphs in graphs with bounded metric dimension or edge metric dimension.

\begin{thm}\label{mdknn}
The maximum possible value of $n$ for which some graph of metric dimension $\leq k$ contains $K_{n,n}$ as a subgraph is between $2^{\lfloor k/2 \rfloor}-1$ and $3^{k}/2$.
\end{thm}

\begin{proof}
The upper bound $n \leq 3^{k}/2$ follows from Theorem \ref{fordiam} since $K_{n,n}$ has diameter $2$ and $2n$ vertices.

For the lower bound, define $G$ to be the graph obtained from $K_{2^{\lfloor k/2 \rfloor},2^{\lfloor k/2 \rfloor}}$ by adding $2\lfloor k/2 \rfloor$ vertices $u_{1}, \dots, u_{\lfloor k/2 \rfloor}, r_{1}, \dots, r_{\lfloor k/2 \rfloor}$ with edges defined as follows: For each vertex $v$ on the left side of the copy of $K_{2^{\lfloor k/2 \rfloor},2^{\lfloor k/2 \rfloor}}$, label $v$ with a binary string of length $\lfloor k/2 \rfloor$, and do the same with each vertex on the right side of the copy of $K_{2^{\lfloor k/2 \rfloor},2^{\lfloor k/2 \rfloor}}$. For each $i$, add an edge from $u_{i}$ to $v$ on the left side of the copy if the $i^{th}$ digit of $v$ is $0$ and add an edge from $r_{i}$ to $v$ on the right side of the copy if the $i^{th}$ digit of $v$ is $0$.

Let $S$ be the vertex subset consisting of the new vertices and suppose that $d_{x, S} = d_{y, S}$ for vertices $x, y \in V(G)$. If there exists $z \in S$ such that $d(x, z) = 0$, then $x = z$. Otherwise $d(x, z) > 0$ for all $z \in S$, so $x$ was a vertex in the copy of $K_{2^{\lfloor k/2 \rfloor},2^{\lfloor k/2 \rfloor}}$. Since $d_{x, S} = d_{y, S}$, $x$ and $y$ are both adjacent to the same elements of $S$, so $x = y$ unless $x$ and $y$ were both labeled with all-ones binary strings. We delete the vertices on each side that were labeled with all-ones binary strings. This gives a lower bound of $n \geq 2^{\lfloor k/2 \rfloor}-1$.
\end{proof}

\begin{thm}\label{emdbcomp}
The maximum possible value of $n$ for which some graph of edge metric dimension $\leq k$ contains $K_{n,n}$ as a subgraph is between $2^{\lfloor k/2 \rfloor}$ and $3^{k/2}$.
\end{thm}

\begin{proof}
The upper bound $n \leq 3^{k/2}$ follows from Theorem \ref{fordiame} since $K_{n,n}$ has diameter $2$ and $n^2$ edges. For the lower bound, define $G$ to be the same graph as in Theorem \ref{mdknn}, before we deleted the vertices on each side that were labeled with all-ones binary strings.

Let $S$ be the vertex subset consisting of the new vertices and suppose that $d_{e, S} = d_{f, S}$ for edges $e, f \in E(G)$. If there exists $x \in S$ such that $d(e, x) = 0$, then $e$ is adjacent to $x$, so $e$ was an added edge. Let $v$ be the other vertex besides $x$ in the edge $e$. Note that for all $t \in S$ such that $t \neq x$, we have $d(e, t) = \min(2, d(v, t))$. There are $2^{\lfloor k/2 \rfloor-1}$ edges adjacent to $x$, and each such edge $e'$ satisfies $d(e', t) = \min(2, d(v', t))$ for all $t \in S$ such that $t \neq x$, where $v'$ is the other vertex in edge $e'$ besides $x$. Thus each of the edges adjacent to $x$ has distinct distance vectors, so $f = e$.

Now suppose that there does not exist $x \in S$ such that $d(e, x) = 0$. Thus $e$ was an edge from the copy of $K_{2^{\lfloor k/2 \rfloor},2^{\lfloor k/2 \rfloor}}$. Since $d_{e, S} = d_{f, S}$, the left vertices in $e$ and $f$ are adjacent to the same elements of $S$, and the right vertices in $e$ and $f$ are adjacent to the same elements of $S$. Thus $e$ and $f$ have the same left and right vertices, so $e = f$. This gives a lower bound of $n \geq 2^{\lfloor k/2 \rfloor}$.
\end{proof}

Using the pattern avoidance bounds, we have several corollaries about the number of edges, chromatic number, and degeneracy.

\begin{cor}
The number of edges in a connected graph of order $n$ and metric dimension $k$ is at most $(3^k-1) n / 2$. 
\end{cor}

\begin{proof}
The maximum possible degree of any vertex in a connected graph of metric dimension $k$ is at most $3^k-1$ by Theorem \ref{mdstar}, and the number of edges is half the sum of the degrees.
\end{proof}

The same reasoning also yields the following corollary, using Theorem \ref{emdstar} instead of Theorem \ref{mdstar}.

\begin{cor}
The number of edges in a connected graph of order $n$ and edge metric dimension $k$ is at most $2^{k-1} n$.
\end{cor}

\begin{cor}
The maximum possible chromatic number of any graph of metric dimension $\leq k$ is between $2^k$ and $3^k$. 
\end{cor}

\begin{proof}
For the upper bound, list the vertices of the graph in any order. Greedily color them with $3^k$ colors, using a free color for each successive vertex since it has at most $3^k-1$ neighbors. For the lower bound, note that we showed in Theorem \ref{mdcomp} that there exists graphs of metric dimension $k$ that contain $K_{2^{k}}$ as a subgraph. 
\end{proof}

\begin{cor}
The maximum possible degeneracy of any graph of metric dimension $\leq k$ is between $2^{k-1}$ and $3^k-1$. 
\end{cor}

\begin{proof}
For the upper bound, note that every subgraph contains a vertex of degree at most $3^k-1$ for the stronger reason that every vertex in the graph has degree at most $3^k-1$. For the lower bound, observe that the graph in Theorem \ref{mdcomp} has degeneracy at least its minimal degree, which is $2^{k-1}$.
\end{proof}

\begin{cor}
The maximum possible degeneracy of any graph of edge metric dimension $\leq k$ is between $2^{\lfloor k/2 \rfloor-1}$ and $2^k$. 
\end{cor}

\begin{proof}
For the upper bound, note that every subgraph contains a vertex of degree at most $2^k$. For the lower bound, observe that the graph in Theorem \ref{emdbcomp} has degeneracy at least its minimal degree, which is $2^{\lfloor k/2 \rfloor-1}$.
\end{proof}

\section{Miscellaneous results}\label{mr}

Kelenc et al. \cite{kel} proved that $2$-dimensional grids have edge metric dimension at most $2$ and that the $d$-dimensional cube has edge metric dimension at most $d$. The result below generalizes both of these results.

\begin{thm}
The $n$-dimensional grid graph $\prod_{i = 1}^{n} P_{r_{i}}$ has edge metric dimension at most $n$.
\end{thm}

\begin{proof}
We prove this for the case when $r_{i} > 1$ for all $i$, which suffices to imply the theorem. We identify vertices with points in $n$-dimensional space with integer coordinates such that the $i^{th}$ coordinate is between $0$ and $r_i -1$ inclusive. There is an edge between two points only if those points agree on all but one coordinate, and the points have a difference of $1$ in that coordinate. Let $S$ be the set consisting of the $n$ vertices $(0, \dots, 0)$, $(r_1-1, 0, \dots, 0)$, $(0, r_2-1, 0, \dots, 0)$, $\dots$, $(0, \dots, 0, r_{n-1}-1, 0)$, where each vertex has $n$ coordinates. 

We write edges of the grid graph in the form $(x_1, \dots, x_n) (y_1, \dots, y_n)$, where $x_i \leq y_i \leq x_i+1$ for all $i$. Suppose that two edges $(a_1, \dots, a_n) (b_1, \dots, b_n)$ and $(c_1, \dots, c_n) (d_1, \dots, d_n)$ have the same distance vector with respect to $S$.

By definition, $\sum_{i = 1}^{n} a_{i} = \sum_{i = 1}^{n} c_{i}$ for the coordinate of the distance vector corresponding to the vertex $(0, \dots, 0)$. Moreover, $r_{j}-1-a_j-b_j+\sum_{i = 1}^{n} a_{i} = r_{j}-1-c_j-d_j+\sum_{i = 1}^{n} c_{i}$ for all $1 \leq j \leq n-1$.

If we subtract the first equation from each of the other equations, we obtain $a_j + b_j = c_j + d_j$ for all $1 \leq j \leq n-1$. Since $a_j \leq b_j \leq a_j+1$ and $c_j \leq d_j \leq c_j+1$, we have $a_j = c_j$ and $b_j = d_j$ for all $1 \leq j \leq n-1$.

Combining the last fact with the first equation, we obtain $a_n = c_n$. Thus $b_n = d_n$, or else $(a_1, \dots, a_n) (b_1, \dots, b_n)$ or $(c_1, \dots, c_n) (d_1, \dots, d_n)$ would not be an edge. We have proved that $(a_1, \dots, a_n) (b_1, \dots, b_n) = (c_1, \dots, c_n) (d_1, \dots, d_n)$, so $S$ is a resolving set for the edges of $\prod_{i = 1}^{n} P_{r_{i}}$.
\end{proof}

Zubrilina gave the following characterization of graphs of order $n$ and edge metric dimension $n-1$ and asked for a characterization of graphs of order $n$ and edge metric dimension $n-2$. 

\begin{thm}\label{char1} \cite{zu}
Let $G(V,E)$ be a graph with $|V| = n$. Then $\edim(G) = n-1$ if and only if for any distinct $v_1, v_2 \in V$ there exists $u \in V$ such that $\left\{v_1,u\right\} \in E$, $\left\{v_{2},u\right\} \in E$, and $u$ is adjacent to all non-mutual neighbors of $v_1, v_2$. 
\end{thm}

We provide a characterization of the graphs of order $n$ and edge metric dimension $\geq n-2$ below, using a similar method to the one in Zubrilina's proof.

\begin{thm}\label{char2}
If $G$ is a connected graph of order $n$, then $\edim(G) \geq n-2$ if and only if for all triples of vertices from $V(G)$ there exists an ordering $v_1, v_2, v_3$ of the triple such that
\begin{enumerate}
\item $v_3$ is adjacent to both $v_1$ and $v_2$ and all non-mutual neighbors of $v_1, v_2$ are adjacent to $v_3$, or
\item  for some $u \in V(G)-\left\{v_1, v_2, v_3\right\}$, $u$ is adjacent to both $v_1$ and $v_2$, all non-mutual neighbors of $v_1, v_2$ in $V(G)-\left\{v_1, v_2, v_3\right\}$ are adjacent to $u$, and any $x \in V(G)-\left\{v_1, v_2, v_3\right\}$ that satisfies $d(x, v_2) > d(x, v_1) = 2$ or $d(x, v_1) > d(x, v_2) = 2$ also satisfies $d(x, u) \leq 2$.
\end{enumerate}
\end{thm}

\begin{proof}
We first show that graphs of order $n$ with edge metric dimension $\geq n-2$ have the above properties. Suppose that $G$ is a graph of order $n$ with $\edim(G) \geq n-2$. Therefore for any distinct triple $T$ of vertices from $V(G)$, the set $V(G)-T$ is not a resolving set for the edges of $G$. Fix a $T$ and let $S = V(G)-T$. Since $S$ is not a resolving set for the edges of $G$, there must exist two distinct edges $e \neq f$ such that $d_{e, S} = d_{f, S}$. If there was a vertex in $S$ that is adjacent to exactly one of $e$ or $f$, then $e$ and $f$ would have different distance vectors with respect to $S$. Thus there exists an ordering $v_1, v_2, v_3$ of $T$ such that $e = \left\{v_1, v_3\right\}$ and $f =  \left\{v_2, v_3\right\}$, or there exists $u \in S$ such that $e = \left\{v_1, u\right\}$ and $f =  \left\{v_2, u\right\}$.

In the first case where $e = \left\{v_1, v_3\right\}$ and $f =  \left\{v_2, v_3\right\}$, suppose that some vertex $x$ is a non-mutual neighbor of $v_1, v_2$. Thus $x \in S$, so $x$ must be adjacent to $v_3$ or else $e$ and $f$ would have different distance vectors with respect to $S$.

In the second case where $e = \left\{v_1, u\right\}$ and $f =  \left\{v_2, u\right\}$, suppose that some vertex $x \in S$ is a non-mutual neighbor of $v_1, v_2$. Then $x$ must be adjacent to $u$ or else $e$ and $f$ would have different distance vectors with respect to $S$. 

Moreover for the second case, suppose that some vertex $y \in S$ satisfies $d(y, v_2) > d(y, v_1) = 2$ or $d(y, v_1) > d(y, v_2) = 2$. Since $e$ and $f$ have the same distance vector with respect to $S$, we must have $d(y, u) \leq \min(d(y, v_1), d(y, v_2)) = 2$. This completes the first direction of the proof.

For the other direction, we will show that any connected $n$-vertex graph $G$ with the properties listed above has edge metric dimension at least $n-2$. Let $S$ be a set of $n-3$ vertices from $V(G)$, and suppose that $T$ is the triple of vertices not in $S$. By assumption, there exists an ordering $v_1, v_2, v_3$ of $T$ such that
\begin{enumerate}
\item $v_3$ is adjacent to both $v_1$ and $v_2$ and all non-mutual neighbors of $v_1, v_2$ are adjacent to $v_3$, or
\item  for some $u \in S$, $u$ is adjacent to both $v_1$ and $v_2$, all non-mutual neighbors of $v_1, v_2$ in $S$ are adjacent to $u$, and any $x \in S$ that satisfies $d(x, v_2) > d(x, v_1) = 2$ or $d(x, v_1) > d(x, v_2) = 2$ also satisfies $d(x, u) \leq 2$.
\end{enumerate}

In the first case, the edges $\left\{v_1, v_3\right\}$ and $\left\{v_2, v_3\right\}$ have the same distance vector with respect to $S$ because all non-mutual neighbors of $v_1, v_2$ are adjacent to $v_3$. For the second case, suppose for contradiction that there exists $x \in S$ such that $d(x, \left\{v_1, u\right\}) \neq d(x, \left\{v_2, u\right\})$. 

Without loss of generality, suppose that $d(x, \left\{v_1, u\right\}) < d(x, \left\{v_2, u\right\})$. Then $d(x, v_1) < \min(d(x, u), d(x, v_2))$. There is a vertex $y \in S$ on a minimal path $P$ from $v_1$ to $x$ that is the closest vertex in $S \cap P$ to $v_1$. Thus either $\left\{y, v_1\right\} \in E(G)$ and $\left\{y, v_2\right\} \not \in E(G)$, or $\left\{y, v_1\right\} \not \in E(G)$ and $\left\{y, v_3\right\} \in E(G)$ and $\left\{v_1, v_3\right\} \in E(G)$ and $d(y, v_2) > 2$.

If $\left\{y, v_1\right\} \in E(G)$ and $\left\{y, v_2\right\} \not \in E(G)$, then $y$ is a non-mutual neighbor of $v_1, v_2$, so $\left\{y, u\right\} \in E(G)$, which contradicts the fact that $d(x, v_1) < d(x, u)$. If $\left\{y, v_1\right\} \not \in E(G)$ and $\left\{y, v_3\right\} \in E(G)$ and $\left\{v_1, v_3\right\} \in E(G)$ and $d(y, v_2) > 2$, then $d(y, v_1) = 2$, so $d(y, u) \leq 2$, which again contradicts that $d(x, v_1) < d(x, u)$.
\end{proof}

Note that if $P_n$ denotes the property equivalent to being a connected $n$-vertex graph of edge metric dimension $n-1$ in Theorem \ref{char1} and $Q_n$ denotes the property equivalent to being a connected $n$-vertex graph of edge metric dimension $\geq n-2$ in Theorem \ref{char2}, then $\neg P_n \wedge Q_n$ is equivalent to being a connected $n$-vertex graph of edge metric dimension $n-2$.

We bound the diameter of connected $n$-vertex graphs of edge metric dimension $n-2$ below, using the characterization proved in the last theorem. 

\begin{thm}
If $G$ is a connected $n$-vertex graph such that $\edim(G) = n-2$, then $G$ has diameter at most $5$.
\end{thm}

\begin{proof}
Suppose for contradiction that $G$ has diameter at least $6$. Then there exist vertices $v_1, v_2, v_3, v_4, v_5, v_6, v_7 \in V(G)$ which form a shortest path (in order) between $v_1$ and $v_7$. Thus $d(v_1, v_4) = 3$, $d(v_4, v_7) = 3$, and $d(v_1, v_7) = 6$. This contradicts the fact from the last theorem that in any connected $n$-vertex graph of edge metric dimension $n-2$, among any three vertices there must exist two vertices $x, y$ such that $d(x, y) \leq 2$. 
\end{proof}

Part of the characterization for $n$-vertex graphs of edge metric dimension $\geq n-2$ can be generalized to graphs of edge metric dimension $n-k$ to give an upper bound of $3k-1$ on the diameter.

\begin{lem}
If $G$ is a connected graph of order $n$ with $\edim(G) \geq n-k$, then for all $(k+1)$-tuples $T$ of vertices from $V(G)$, there exists an ordering $v_1, \dots, v_{k+1}$ of the vertices in $T$ such that $d(v_1, v_2) \leq 2$.
\end{lem}

\begin{proof}
Suppose that $G$ is a graph of order $n$ with $\edim(G) = n-k$. Therefore for any distinct $(k+1)$-tuple $T$ of vertices from $V(G)$, the set $V(G)-T$ is not a resolving set for the edges of $G$. As in the proof for $k = 2$, fix a $T$ and let $S = V(G)-T$. Since $S$ is not a resolving set for the edges of $G$, there must exist two distinct edges $e \neq f$ such that $d_{e, S} = d_{f, S}$. If there was a vertex in $S$ that is adjacent to exactly one of $e$ or $f$, then $e$ and $f$ would have different distance vectors with respect to $S$. Thus there exists an ordering $v_1, \dots, v_{k+1}$ of $T$ such that $e = \left\{v_1, v_3\right\}$ and $f =  \left\{v_2, v_3\right\}$, or $e = \left\{v_1, v_2\right\}$ and $f =  \left\{v_3, v_4\right\}$, or there exists $u \in S$ such that $e = \left\{v_1, u\right\}$ and $f =  \left\{v_2, u\right\}$. In each case, $d(v_1, v_2) \leq 2$.
\end{proof}

\begin{thm}
If $G$ is a connected $n$-vertex graph such that $\edim(G) = n-k$, then $G$ has diameter at most $3k-1$.
\end{thm}

\begin{proof}
Suppose for contradiction that $G$ has diameter at least $3k$. Then there exist vertices $v_1, \dots, v_{3k+1} \in V(G)$ which form a shortest path (in order) between $v_1$ and $v_{3k+1}$. Define $T = \left\{v_{3i+1}: 0 \leq i \leq k \right\}$. Observe that $d(x, y) > 2$ for all $x, y \in T$ such that $x \neq y$. This contradicts the fact from the last lemma that in any connected $n$-vertex graph of edge metric dimension $n-k$, among any $k+1$ vertices there must exist two vertices $x, y$ such that $d(x, y) \leq 2$. 
\end{proof}

\section{Conclusion and open problems}

We have found the maximum value of $n$ for which some graph of metric dimension $k$ contains the complete graph $K_{n}$ and the maximum value of $n$ for which some graph of edge metric dimension $k$ contains $K_{1,n}$. However, for all of the other bounds in this paper, there is a gap between the upper and lower bounds.

We have characterized $n$-vertex graphs with edge metric dimension $n-2$ and proved that they have diameter at most $5$. We also proved that $n$-vertex graphs with edge metric dimension $n-k$ have diameter at most $3k-1$. Two natural problems are to characterize $n$-vertex graphs with edge metric dimension $n-k$ for fixed $k > 2$, and to find the exact value for the maximum possible diameter of an $n$-vertex graph of edge metric dimension $n-k$.

One quantity that we did not bound in this paper is the maximum value of $n(k)$ for which some graph of edge metric dimension $k$ contains the complete graph $K_{n(k)}$. Since the edge metric dimension of $K_{n}$ is $n-1$, there is a lower bound of $n(k) \geq k+1$. Since the complete subgraph can have at most $2^{k}$ edges by the pattern avoidance diameter bounds in Section \ref{btdd}, there is an upper bound of $n(k) = O(2^{k/2})$. What is the exact value of $n(k)$? 

We also have found an upper bound of $d$ on the edge metric dimension of $d$-dimensional grid graphs $\prod_{i = 1}^{d} P_{r_{i}}$. A related open problem is to find the maximum possible edge metric dimension for graphs of the form $\prod_{i = 1}^{d} C_{r_{i}}$.

\section*{Acknowledgement}

The author thanks the anonymous referee for many helpful comments and correcting errors in the paper.


\begin{thebibliography}{}
\bibitem{ch} Gary Chartrand, Linda Eroh, Mark A. Johnson, and Ortrud R. Oellermann. Resolvability in graphs and the metric dimension of a graph. Discrete Appl. Math. 105 (2000) 99-113.
\bibitem{h} Frank Harary and R. A. Melter. On the metric dimension of a graph. Ars Combin. 2 (1976) 191-195.
\bibitem{kel} A. Kelenc, N. Tratnik, and I. G. Yero. Uniquely identifying the edges of a graph: the edge metric dimension. Discrete Appl. Math. 251 (2018) 204-220.
\bibitem{kh} Samir Khuller, Balaji Raghavachari, and Azriel Rosenfeld. Landmarks in graphs. Discrete Appl. Math. 70 (1996) 217-229.
\bibitem{pr} R. A. Melter, I. Tomescu, Metric bases in digital geometry, Computer Vision, Graphics, and Image Processing 25 (1984) 113-121.
\bibitem{li} P. J. Slater, Leaves of trees, Proceeding of the 6th Southeastern Conference on Combinatorics, Graph Theory, and Computing, Congr. Numer. 14 (1975) 549-559.
\bibitem{exmd} M. Hernando, M. Mora, I. Pelayo, C. Seara, and D. Wood. Extremal Graph Theory for Metric Dimension and Diameter. Electron. J. Combin. 17 (2010) R30.
\bibitem{waset} G. Sudhakara and A. Kumar. Graphs with metric dimension two - a characterization. World Academy of Science, Engineering and Technology: Int. J. Comput. Math. 3 (2009) 1128-1133.
\bibitem{zu} N. Zubrilina. On the edge dimension of a graph. Discrete Math. 341 (2018) 2083-2088.
\end{thebibliography}
\end{document}